\documentclass{amsart}

\usepackage{amsmath}
\usepackage{amssymb}
\usepackage[alphabetic]{amsrefs}
\usepackage{hyperref}
\usepackage{mathrsfs}
\usepackage{tikz-cd}
\usetikzlibrary{arrows}

\newtheorem{theorem}{Theorem}[section]
\newtheorem{lemma}[theorem]{Lemma}
\newtheorem{proposition}[theorem]{Proposition}
\newtheorem{corollary}[theorem]{Corollary}

\theoremstyle{definition}

\theoremstyle{remark}
\newtheorem{remark}[theorem]{Remark}

\numberwithin{equation}{section}

\def\Spec{\mathop{\mathrm{Spec}}}

\def\O{\mathscr{O}}
\def\H{\mathcal{H}}
\def\WO{{W\Omega_X^N}}

\def\WOn{{W_n\Omega_X^N}}

\def\fa{{\text{~for any~}}}

\def\tensor{\mathop{\otimes}\limits}

\def\sHom{\mathop{\mathcal{H}\! \mathit{om}}\nolimits}
\def\Hom{\mathop{\mathrm{Hom}}\nolimits}

\makeatletter
\newcommand{\proofstep}[1]{%
  \par
  \addvspace{\medskipamount}
  \textit{#1\@addpunct{:}}\enspace\ignorespaces
}
\makeatother

\begin{document}

\title{
    Vanishing and a Counterexample for Witt Divisorial Sheaves
}


\author{Niklas Lemcke}
\subjclass[2010]{14F30, 14F17, 14J26}
\keywords{Witt sheaves, Kodaira vanishing theorem, Kawamata-Viehweg vanishing theorem, positive characteristic}

\address{Department of Mathematics, School of Science and Engineering, Waseda University,
Ohkubo 3-4-1, Shinjuku, Tokyo 169-8555, Japan}
\email{numberjedi@akane.waseda.jp}

\begin{abstract}
    First we refine the duality theory for Witt divisorial sheaves on smooth projective varieties over a perfect field of positive characteristic.
    Building on previous work \cite{Lemcke}, we remove the residual derived limit to obtain a cleaner isomorphism.
    As an application, we prove a Ramanujam--type vanishing theorem for Witt divisorial sheaves of nef and big divisors on surfaces.
    Finally, we show that a surface constructed by Langer \cite{Langer} with a divisor constructed by Cascini and Tanaka \cite{CT} gives a counterexample to Kawamata--Viehweg--type vanishing of Witt divisorial sheaves in dimension two.
\end{abstract}

\maketitle

\tableofcontents

\section*{Introduction}
    It is well--known that Kodaira Vanishing and its generalizations do not hold in positive characteristic. 
    One possible approach to overcome this is to use Witt vectors to move from a positive characteristic setting to a characteristic zero setting, losely speaking. This can be attempted in at least two different ways.

    Given an effective Cartier disivor $D$ on a projective variety $X$, divisorial sheaves $\O_X(-D)$ can be thought of in two equivalent ways:
    \begin{itemize}
        \item as the ideal sheaf $\mathscr I_D$ of $\O_X$,
        \item as the invertible $\O_X$--module locally generated by the local equations of $D$.
    \end{itemize}
    In the case of Witt vectors, however, these two characterizations are no longer equivalent.

    Recent work exploring the first option includes papers by Ren and Ruelling~\cite{RR} as well as Baudin~\cite{Baudin}.
    This approach has the advantage that the endomorphisms $V$ and $F$ naturally extend to the Witt sheaf.
    In \cite{Lemcke} we had to consider the infinite direct sum $\bigoplus_\mathbb Z^t W\O(p^tD)$ to accomplish this for the second approach.

    This paper continues to discuss the second approach, based on previous work by Tanaka~\cite{Tanaka} and the author~\cite{Lemcke}.
    This approach has the advantage of being more easily generalized beyond effective divisors.
    Tanaka proposed a Kodaira--like vanishing theorem using the second approach, which holds for ample divisors in positive characteristic.

    \begin{theorem}[\cite{Tanaka}*{Corollary 4.16, Theorem 4.17, Theorem 5.3 Step 5}]\label{thmTanaka}
        Let $k$ be a perfect field of positive characteristic $p$, and $X \xrightarrow \phi \text{Spec }k$ be an $N$--dimensional smooth projective variety.
        If $A$ is an ample $\mathbb Q$--divisor with simple normal crossing support on $X$, then
        \renewcommand{\theenumi}{\roman{enumi}}
        \begin{enumerate}
            \item 
                \begin{itemize}
                    \item $H^j(X, W\O_X(-A))$ is $p^t\text{--torsion}, \text{ for some } t, \fa j < N$.
                    \item $H^j(X, W\O_X(-sA))$ is torsion free for large enough $s$, all $j$.
                \end{itemize}
            \item 
                \begin{itemize}
                    \item $R^i\phi_*\sHom_{W\O_X}(W\O_X(-A),\WO) = 0$ for $0<i$ if $A$ is a $\mathbb Z$--divisor.
                    \item the above term is at most torsion if $A$ is a $\mathbb Q$--divisor.
                \end{itemize}
        \end{enumerate}
    \end{theorem}

    In \cite{Lemcke}*{Theorem 3.10} we established the isomorphism 
    \[
        R\phi_*R\lim_n \sHom_{W\O_X}(\omega(D), \WOn)
        \cong
        R\Hom_\omega(R\phi_*\omega(D), \check\omega[-N]),
    \]
    which allowed us to recover (ii) of Theorem \ref{thmTanaka}
    from the easier to prove (i) (cf. \cite{Lemcke}*{Corollary 3.12}).

    The remaining derived limit functor on the left hand side was unsatisfactory. 
    In this paper we show that the higher derived limits indeed vanish for any $\mathbb Q$--divisors with SNC support (cf. Proposition \ref{prop1}), which leads to the improved statement of the first theorem:
    \begin{theorem}[cf. Theorem \ref{thm}]\label{thmIntro}
        Let $k$ be a perfect field of characteristic $0<p$, and $X \xrightarrow \phi \text{Spec }k$ be an $N$--dimensional smooth projective variety.
        If $D$ is a $\mathbb Q$--divisor with simple normal crossing support, then
        \[
            R\phi_*\sHom_{W\O_X}(\omega(D), \WO)
            \cong R\Hom_\omega(R\phi_*\omega(D), \check\omega[-N]).
        \]
        where we write $\omega$ for the Cartier--Dieudonn\'e ring $W_\sigma[F,V]$, $\check\omega$ for a certain left--$\omega$--module
        and $\omega(D) := \bigoplus_{t \in \mathbb Z} W\O_X(p^tD)$.
    \end{theorem}
    Note that we consider $\omega(D)$ instead of $W\O_X(D)$ since, unlike $W\Omega_X^\bullet$ in \cite{Ekedahl}, the latter does not allow for an appropriate $\omega$--module structure.
    We then Apply Theorem \ref{thm} to show that for ample $\mathbb Q$--divisors, the torsion in Theorem \ref{thmTanaka}(ii) is in fact zero.
    \begin{corollary}[cf. Corollary \ref{cor}]\label{corIntro}
        Let $X$ be as in Theorem \ref{thmIntro}.
        Suppose that $R^j\Gamma_X(W\O_X(p^tD))$ is $V$--torsion free for $j\leq N$ and large enough $t$.
        Then
        \[
            R\Hom_\omega(R\phi_*\omega(D), \check\omega[-N])
            \cong 
            \Hom_\omega(R\phi_*\omega(D), \check\omega[-N]).
        \]
        If furthermore $R^j\phi_*\omega(D)\otimes_\mathbb Z\mathbb Q = 0$ for $j<N$, then 
        \[
            R^i\phi_*(\sHom_{W\O_X}(W\O_X(D), \WO)) = 0
        \]
        for any $0<i$.
        In particular, by Theorem \ref{thmTanaka}(i) this is the case when $-D$ is ample.
    \end{corollary}
    Corollary \ref{corIntro} makes a notable distinction between $V$--torsion and $p$--torsion. 
    The proof of Theorem \ref{thmTanaka} relies on Serre vanishing, which for ample divisors $A$ and sufficiently large $t$ guarantees the vanishing of $H^{j<N}(X, W\O_X(-p^tA))$.
    Corollary \ref{corIntro} shows that, in order for the vanishing theorem of Theorem \ref{thmTanaka}(ii) to hold for a SNC $\mathbb Q$--divisor $D$, such vanishing for large multiples of $D$ is not necessary (cf. Remark \ref{rmk}).

    We would like to gauge the type of settings in which the vanishing of Witt divisorial sheaves may hold. 
    In order to do so, sections 3 and 4 focus on the simpler special case of surfaces.
    Theorem \ref{thmKVV} proves that vanishing holds for nef and big $\mathbb Q$--divisors on a surface.
    \begin{theorem}[cf. Theorem \ref{thmKVV}]
        For $X$ a smooth surface over a field $k$ of positive characteristic and $D$ a nef and big $\mathbb Q$--divisor with SNC support on $X$, 
        \begin{enumerate}
            \item $H^1(X,W\O_X(-D))_\mathbb Q=0$,
            \item $H^1(X,\sHom_{W\O_X}(W\O_X(-D),\WO)) = 0$.
        \end{enumerate}
    \end{theorem}
    In section \ref{sec4} we show that the counterexample to Kawamata--Viehweg Vanishing of \cite{CT} fails vanishng in the Witt divisoral case too.
    \begin{theorem}[cf. Theorem \ref{example}]
        There exists a klt surface pair $(X,\Delta)$ and an integral divisor $B$ on $X$, with $X$ smooth and $B-\Delta$ nef and big, such that 
        \[
            H^1(X,W\O_X(-B))_\mathbb Q \neq 0.
        \]
    \end{theorem}

    \subsection*{Acknowledgements}
        I would like to thank my advisor Hajime Kaji for his continued support and useful suggestions. 
        I would also like to thank Kay R\"ulling for a fruitful conversation, Shunsuke Takagi for his helpful comments and pointing out an alternative proof of Lemma \ref{lemSurfaceNaBSerre}, as well as Hiroyuki Ito, Noriyuki Suwa and Tomohide Terasoma for their time and useful comments.


\section{Notation}
    Fix the following notations and conventions:
    \begin{itemize}
        \item A variety over $k$ is a separated integral scheme of finite type over $k$.
        \item Throughout this paper we define $X \xrightarrow \phi S=\Spec k$, where $k$ is a perfect field of characteristic $p>0$, and $X$ is assumed to be a smooth projective variety.
        \item If $C$ is a complex, $C[i]$ denotes $C$ shifted by $i$ in complex degree.
        \item If $M_n$ is an inverse system, then $\lim_n M_n$ denotes the inverse limit.
        \item Given a ring $A$, $W(A)$ are the Witt vectors of $A$. If $k$ is the base field, we may write $W$ instead of $W(k)$ for the sake of brevity.
        The truncated Witt vectors are denoted by $W_n(A)$ or $W_n$ respectively, for $n \in \mathbb N$.
        $F$ and $V$ denote the Frobenius and {\it Verschiebungs} maps on $W(A)$, respectively.
        \item The notion of Witt vectors can be sheafified to yield sheaves $W\O_X$, $W_n\O_X$, $W\O_X(D)$ and $W_n\O_X(D)$ 
        for $n \in \mathbb N$ and a $\mathbb Q$--divisor $D$ on $X$.
        The truncated versions of these sheaves are coherent $W_n\O_X$--modules (cf. \cite{Tanaka}*{Prop. 3.8}).
        For more details the reader may refer to \cite{Tanaka}*{Sect. 2, 3} or \cite{Lemcke}*{Def. 2.1, 2.2}.
        The ringed space $(X, W_n\O_X)$ is known to be a noetherian scheme (cf. \cite{Illusie}).
        \item $W_\bullet\Omega_X^\bullet$ denotes the De Rham--Witt complex of $X$ (cf. \cite{Illusie}).
        \item For the sake of brevity we write $\omega$ for the (non--commutative) {\it Cartier--Dieudonn\'e ring} $W_\sigma[F,V]$, 
        that is the $W$--algebra generated by $V$ and $F$, subject to the relations 
        \[
            aV=V\sigma(a), Fa=\sigma(a)F \text{ for any } a \in W; VF=FV=p,
        \]
        where $\sigma$ is the Frobenius map on $W$, induced by that on $k$.
        Its truncation at $n$ is denoted by $\omega_n$.
    \end{itemize}

\section{Duality}
    Let $X$ be a smooth projective variety of dimension $N$ over a perfect field $k$ of positive characteristic $p$.
    Recall from \cite{Lemcke}[Propositions 2.3 and 2.4] that 
    \begin{equation}\label{iso1}
        \sHom_{W_n\O_X}(W_n\O_X(D),\WOn) \cong \sHom_{W\O_X}(W\O_X(D), \WOn)
    \end{equation}
    and
    \begin{equation}\label{iso2}
        \sHom_{W_m\O_X}(W_m\O_X(D),W_m\Omega_X^N) \cong R\sHom_{W_n\O_X}(W_m\O_X(D),\WOn),
    \end{equation}
    for $m\leq n$.
    We will use these two isomorphisms liberally throughout the following.
    \begin{lemma}\label{lem1}
        Let $D$ be a $\mathbb Q$--divisor on $X$. Suppose there exists a basis $\mathscr B$ for the topology of $X$ such that, 
        for some $t_0 \in \mathbb N, t_0 \leq t$ and any $U \in \mathscr B$, the map
        \[
            \Gamma_U(\sHom(W\O_X(p^tD),W_{n+1}\Omega_X^N)) \rightarrow \Gamma_U(\sHom(W\O_X(p^tD),W_{n}\Omega_X^N))
        \]
        induced by $W_{n+1}\Omega_X^N \xrightarrow R \WOn$ is surjective.
        
        Then 
        \[
            \sHom_{W\O_X}(W\O_X(D),\WO) \cong R\lim_nR\sHom_{W_n\O_X}(W_n\O_X(D),\WOn).
        \]
    \end{lemma}
    \begin{proof}
        The statement follows from
        \[
            R^i\lim_nR\sHom(W_n\O_X(D), W_n\Omega_X^N) = 0 \fa 1\leq i.
        \]

        Let $E_n := \sHom_{W\O_X}(W\O_X(D), \WOn)$.
        By \cite{CR}*{Lemma 1.5.1} we need to show that for $\mathscr B$ a basis for the topology of $X$, $U \in \mathscr B$, the projective system $((H^0(U, E_n))_{1\leq n}, \psi_{ij})$ satisfies the Mittag--Leffler condition (ML), i.e. that the projective system is {\it eventually stable}.
        That is, for any $k \in \mathbb N$, $\exists N_k \in \mathbb N$ such that $\psi_{kn}(H^0(U,E_n)) = \psi_{kl}(H^0(U,E_l))$, for $N_k \leq n,l$.
        (Condition (1) of the above Lemma is satisfied by coherence.)

        For any $x \in X$ fix an affine open subset $U_x \subset X$, and for ease of notation write $\H_j(D) := \Gamma_U\left(\sHom(W_j\O(D),\WOn)\right)$.
        Note that, thanks to \ref{iso1} and \ref{iso2}, for $j \leq n$ we actually know that
        \[
             \H_j(D) \cong \Gamma_U\left(\sHom(W_j\O(D),W_j\Omega_X^N)\right)
             \cong H^0(U,E_j).
        \]

        Suppose now that $t_0 \leq t$.
        Then by assumption
        \[
            \H_j(p^tD) \xrightarrow {\psi_{ij}} \H_i(p^tD)
        \]
        is surjective for $i\leq j$.
        For $D$ we have an exact sequence
        \[
            0 \rightarrow F_*W_{j-1}\O(pD) \xrightarrow V W_j\O(D) \rightarrow \O(D) \rightarrow 0.
        \]
        Because $H^1(U,\sHom(\O_X(D), \WOn)) \cong H^1(U,E_1) = 0$ due to coherence, this sequence yields a surjective map 
        \[
            \H_j(D) \xrightarrow {V^*} \H_{j-1}(pD).
        \]
        Fix $k$ and set $n:= k+t+1$.
        Since $FV=VF=p$, the below diagram is commutative.

        \begin{equation}\label{diagIndPi}
            \begin{tikzcd}
                & & & \H_n(D) \arrow[two heads]{d}{V^*} \\
                & & \H_{n-1}(D) \arrow[two heads]{d}{V^*}  
                    & \H_{n-1}(pD) \arrow{l}{F^*} \arrow[two heads]{d}{V^*} \\
                & & \vdots \arrow[two heads]{d}{V^*} & \vdots \arrow[two heads]{d}{V^*} \\
                & \dots & \H_{k+1}(p^{t-1}D) \arrow{l}{F^*} \arrow[two heads]{d}{V^*}  
                    & \H_{k+1}(p^{t}D) \arrow{l}{F^*} \arrow[two heads]{dl}{p^*} 
                    \\
                \H_k(D) & \dots \arrow{l}{F^*} & \H_k(p^{t}D) \arrow{l}{F^*} 
            \end{tikzcd}
        \end{equation}

        The diagonal maps $ F^* \circ V^* $ are exactly the transition maps $\psi_{ij}$ of the projective system 
        $(H^0(U,E_n))_{1\leq n}$.
        We see this once we compose $R$ with the isomorphisms \ref{iso1} and \ref{iso2} appropriately.
        The $\psi_{ij}$ are given via the chain of isomorphisms
        \[
        \begin{aligned}
            \sHom(W_{n+1}\O_X(D), W_{n+1}\Omega_X^N) 
            & \xrightarrow {R_{n+1}^*} \sHom(W\O_X(D), W_{n+1}\Omega_X^N) \\
            & \xrightarrow {R\circ} \sHom(W\O_X(D), W_n\Omega_X^N) \\
            & \xrightarrow {(R_n^*)^{-1}} \sHom(W_n\O_X(D), W_n\Omega_X^N) \\
            & \xrightarrow {p\circ} \sHom(W_{n}\O_X(D), W_{n+1}\Omega_X^N).
        \end{aligned}
        \]
        So $\psi_{nn+1}$ sends $f$ to 
        \[
        \begin{aligned}
            g &= p\circ R_n^{n+1} \circ f \circ (R^{n+1}_n)^{-1} \\
            &= p \cdot f \circ (R^{n+1}_n)^{-1} \\
            &= p \cdot (f \circ (R^{n+1}_n)^{-1}) \\
            &= f \circ p
        \end{aligned}
        \]
        That is $\psi_{nn+1} = p^*$.
        Here we differentiate between the notations $p$ and $p\cdot$, 
        which stand for either of the maps $W_n\O_X(D) \xrightarrow p W_{n+1}\O_X(D)$, $W_n\Omega_X^N \xrightarrow p W_{n+1}\Omega_X^N$, 
        or the module operation of multiplication by $p$, respectively.

        Set $N_k := n-1$. Then for $N_k \leq n,l$, $\psi_{kn}(H^0(U,E_k)) = \psi_{kl}(H^0(U,E_l))$, i.e. $\psi$ is eventually stable, 
        hence ML is satisfied and the result follows. 
    \end{proof}

    \begin{proposition}\label{prop1}
        Let $D$ be a $\mathbb Q$--divisor with simple normal crossing support, such that $\ell D$ is integral
        for some $\ell \in \mathbb Z \backslash p\mathbb Z$.
        Let $\mathscr B$ be an affine basis for the topology of $X$, and $U \in \mathscr B$.
        Then the map
        \[
            \Gamma_U(\sHom(W\O_X(D),W_{n+1}\Omega_X^N)) \rightarrow \Gamma_U(\sHom(W\O_X(D),W_{n}\Omega_X^N))
        \]
        induced by $W_{n+1}\Omega_X^N \xrightarrow R \WOn$ is surjective.

        Therefore, by Lemma \ref{lem1}, for any simple normal crossing $\mathbb Q$--divisor $D$,
        \[
            \sHom_{W\O_X}(W\O_X(D),\WO) \cong R\lim_nR\sHom_{W_n\O_X}(W_n\O_X(D),\WOn).
        \] 
    \end{proposition}
    \begin{proof}
        First suppose that $\ell = 1$.
        Apply $\sHom(W_{n+1}\O_X(D), \cdot )$ to the short exact sequence of $W_{n+1}\O_X$--modules (cf. \cite{Illusie})
        \[
            0 \rightarrow \text{gr}^n\WO \rightarrow W_{n+1}\Omega_X^N \rightarrow \WOn \rightarrow 0,
        \]
        where $\text{gr}^n\WO$ is coherent (cf. \cite{Tanaka}*{2.6}).
        Since $W_{n+1}\O_X(D)$ is locally free, the result is another short exact sequence
        \[
        \begin{aligned}
            0 &\rightarrow \sHom_{W_{n+1}\O_X}(W_{n+1}\O_X(D), \text{gr}^n\WO) \\
            &\rightarrow \sHom_{W_{n+1}\O_X}(W_{n+1}\O_X(D), W_{n+1}\Omega_X^N) \\
            &\rightarrow \sHom_{W_{n}\O_X}(W_{n}\O_X(D), \WOn) \rightarrow 0
        \end{aligned}
        \]
        Applying $\Gamma_U$ to that sequence yields yet another short exact sequence (thanks to coherence), and therefore the desired surjection.

        Now assume $1 \leq \ell$.
        Let $\H_n(D)$ be as in \ref{lem1}.
        We use the proof of \cite{Tanaka}*{Theorem 4.15} to show surjectivity of 
        \[
            \H_n(D) \xrightarrow{p^*} \H_{n-1}(D)
        \] 
        for $\mathbb Z_{(p)}$--divisors, 
        that is $\mathbb Q$--divisors $D$ such that $\ell D$ is a $\mathbb Z$--divisor for $\ell \in \mathbb Z \backslash p \mathbb Z$.
        Then we use Lemma \ref{lem1} to reduce the case of arbitrary $\mathbb Q$--divisors to that of $\mathbb Z_{(p)}$--divisors.

        First assume that $D$ is a $\mathbb Z_{(p)}$--divisor. 
        According to \cite{Tanaka}*{Proposition 4.14} there exists a finite surjective $k$--morphism $X'' \xrightarrow h X$,
        such that $h^*D$ is a $\mathbb Z$--divisor.
        Analogue to step 1 of \cite{Tanaka}*{Theorem 4.15}, there is a split surjection 
        \[
            \sHom_{W_n\O_{X}}(h_* W_n\O_{X''}(D_{X''}), W_n\Omega_{X}^N) 
            \rightarrow \sHom_{W_n\O_{X}}(W_n\O_{X}(D), W_n\Omega_X^N),
        \]
        of $W_n\O_X$--modules (push forward the split surjection on $X'$ along $f$).
        Because $D_{X''}$ is a $\mathbb Z$--divisor,
        \[
            \sHom_{W\O_{X''}}(W\O_{X''}(D_{X''}),W_n\Omega_{X''}^N) \xrightarrow{R''} \sHom_{W\O_{X''}}(W\O_{X''}(D_{X''}), W_{n-1}\Omega_{X''}^N)
        \]
        is surjective as well.
        Pushing forward along $h$, applying coherent duality, \cite{Ekedahl}*{Theorem 4.1}, 
        and the fact that $h$ is affine and therefore $R^ih_*\mathscr F = 0$ for $0<i$ and quasi--coherent $\mathscr F$, 
        we obtain another surjection
        \[
            \sHom(h_*W_n\O_{X''}(D_{X''}),\WOn) \xrightarrow{h_*R''} \sHom(h_*W_{n-1}\O_{X''}(D_{X''}), W_{n-1}\Omega_X^N).
        \]
        This yields the following commutative diagram, where $h_*R''$ is still surjective due to coherence.
        \begin{equation*}
            \begin{tikzcd}
                \Gamma_U(\sHom(h_*W_n\O_{X''}(D_{X''}), W_n\Omega_{X}^N)) \arrow[two heads]{d}{h_*R''} \arrow[two heads]{r}{} 
                & \Gamma_U(\sHom(W\O_{X}(D), W_n\Omega_{X}^N)) \arrow{d}{R} \\
                \Gamma_U(\sHom(h_*W_{n-1}\O_{X''}(D_{X''}), W_{n-1}\Omega_{X}^N)) \arrow[two heads]{r}{} 
                & \Gamma_U(\sHom(W\O_{X}(D), W_{n-1}\Omega_{X}^N)),
            \end{tikzcd}
        \end{equation*}
        Because the diagram is commutative, $R$ is surjective. 

        The second statement now follows from $\ref{lem1}$.
    \end{proof}

    \begin{theorem}\label{thm}
        Suppose $D$ is a $\mathbb Q$--divisor with SNC support. Then
        \[
            R\phi_*\sHom_{W\O_X}(\omega(D), \WO)
            \cong R\Hom_\omega(R\phi_*\omega(D), \check\omega[-N]).
        \]
        Here we denote 
        \[
            \begin{aligned}
                &\omega(D) := \bigoplus_{t \in \mathbb Z} W\O_X(p^tD),
                &\check\omega := \lim_n \check\omega_n
            \end{aligned}
        \]
        where, as left--$\omega$--modules,
        \[
            \check\omega_n \cong \bigoplus_{0<i<n} F^iW_{n-i}\oplus \prod_{0\leq j}F^{-j}W_n
        \]
        (cf. \cite{Lemcke}*{Proposition 3.9}).
    \end{theorem}
    \begin{proof}
        Choose $t_0$ such tat $p^{t_0}D$ is a $\mathbb Z_{(p)}$--divisor. 
        Then by Proposition \ref{prop1}, Lemma \ref{lem1} applies to $p^{t_0}D$, and hence also to $D$.
        The statement then immediately follows from \cite{Lemcke}*{Theorem 3.10}, Lemma \ref{lem1} and Proposition \ref{prop1}.
    \end{proof}
    
    \begin{corollary}\label{cor}
        Let $X$ and $D$ be as in Theorem \ref{thm}.
        Suppose that $R^j\Gamma_X(W\O_X(p^tD))$ is $V$--torsion  free for $j\leq N$ and large enough $t$.
        Then
        \[
            R\Hom_\omega(R\phi_*\omega(D), \check\omega[-N])
            \cong 
            \Hom_\omega(R\phi_*\omega(D), \check\omega[-N]).
        \]
        If furthermore $R^j\phi_*\omega(D)\otimes_\mathbb Z\mathbb Q = 0$ for $j<N$, then 
        \[
            R^i\phi_*(\sHom_{W\O_X}(W\O_X(D), \WO)) = 0
        \]
        for any $0<i$.
        In particular, by Theorem \ref{thmTanaka}(i) this is the case when $-D$ is ample.
    \end{corollary}
    \begin{proof}
        This proof is contained in \cite{Lemcke}*{Corollary 3.12}. 
    \end{proof}
    
    \begin{remark}\label{rmk}
        Corollary \ref{cor} gives some additional insight into the composition of $R\phi_*\sHom(W\O_X(D),\WO)$.
        Of particular interest is the distinction between $V$--torsion and $p$--torsion.
        The proofs of vanishing \ref{thmTanaka}(ii) for ample divisors $A$ in \cite{Tanaka} rely on Serre vanishing, which guarantees that the cohomology groups 
        $H^{j<N}(X, W\O_X(-p^tA))$ vanish for large enough $t$.
        Corollary \ref{cor} however does not require the cohomology of $-p^tA$ to vanish for any $t$ --- only to be torsion 
        and to be $V$--torsion free for large enough $t$. 
        There are examples which show that this distinction may be relevant.
        For example, for a supersingular abelian surface $X$, $H^2(X, W\O_X)\cong k[[V]]$ (cf. \cite{Illusie}*{7.1 (b)}) is $V$--torsion free $p$--torsion.

        In fact, by setting $D=0$, Corollary \ref{cor} allows us to recover the $W\Omega_X^2$--column of the table in \cite{Illusie}*{7.1 (b)} from the $W\O_X$--column.
    \end{remark}

\section{Vanishing on Surfaces}\label{sec3}
    Throughout this section we assume $X$ to be a smooth surface over a perfect field $k$ of positive characteristic $p$, and $D$ a nef and big $\mathbb Q$--divisor with SNC support on $X$.

    \begin{lemma}\label{lemSurfaceNaBSerre}
        There exists $n_0 \in \mathbb N$ such that
        \[
            H^1(X,\O_X(-p^nD)) = 0
        \]
        for all $n_0 \leq n$.
    \end{lemma}
    \begin{proof}
        Let $t_1$ be large enough such that $p^{t_1}D$ is a $\mathbb Z_{(p)}$--divisor, i.e. such that there exists $\ell \in \mathbb Z \backslash p \mathbb Z$ such that $\ell p^{t_1}D$ is a $\mathbb Z$--divisor.
        Write $D_1 := p^{t_1}D$ for brevity.
        Then by \cite{Tanaka}*{Thm. 4.14}, there exists a finite surjective morphism of smooth surfaces
        \[
            X'' \xrightarrow{h} X
        \]
        such that $h^*D_1$ is a (nef and big) $\mathbb Z$--divisor.
        By \cite{LM}*{Proposition 2}, there exists $t_0 \in \mathbb N$ such that 
        \[
            \begin{aligned}
                H^1(X'',\O_{X''}(-p^th^*D_1)) &= 0
                \fa t_0 \leq t, \text{ or equivalently} \\
                H^1(X'',\O_{X''}(-p^th^*D)) &= 0 
                \fa t_0 + t_1 \leq t.
            \end{aligned}
        \]

        By \cite{Tanaka}*{Thm. 4.15, step 1}, the natural map
        \[
            \O_X(-p^tD) \xrightarrow{\alpha} h_*\O_{X''}(-p^th^*D)
        \]
        splits.
        (Cf. for example \cite{Tanaka}*{Prop. 4.7 (1); Lem. 4.4 (1)}, or just tensor by $\O_X$, to see that Tanaka's statement applies to the truncated case as well --- in particular to $n=1$).
        Since $h$ is finite, any $R^ih_*$ are zero for $0<i$ and we have an injection
        \[
            H^1(X, \O_X(-p^tD)) \xrightarrow{\alpha}
            H^1(X'', \O_{X''}(-p^th^*D)).
        \]
        Hence
        \[
            H^1(X, \O_X(-p^tD)) = 0 \fa t_0+t_1 \leq t.
        \]
    \end{proof}
    \begin{remark}
        It was pointed out to the author by Shunsuke Takagi that the above Lemma also follows from \cite{TanakaX}*{Theorem 2.6}. 
        This is done by choosing $e,l \in \mathbb N$ such that $p^e(p^l-1)D$ is a $\mathbb Z$--divisor, and then setting
        \[
            B_i := p^{e+i}D, N_i := p^{e+i}(p^l-1)D
        \]
        for each $0 \leq i \leq l-1$.
        Applying \cite{TanakaX}*{Theorem 2.6} to $B_i$ and $N_i$ yields $r_i \in \mathbb N$ such that for for all $r_i \leq t$,
        \[
            H^1\left(X,K_X+\lceil B_i\rceil + (p^{l(t-1)}+\cdots +p^l+1)N_i\right) = 0.
        \]
        Since 
        \[
            \lceil B_i\rceil + (p^{l(t-1)}+\cdots +p^l+1)N_i = \lceil p^{e+lt+i}D\rceil,
        \]
        it follows that after choosing $r := \max_i\{r_i\}$,
        \[
            H^1(X,\O(-p^nD)) \cong H^1(X,K_X+\lceil p^nD \rceil) = 0 \fa e+lr \leq n.
        \]
    \end{remark}

    \begin{lemma}\label{lemSurfaceNaB}
        There exists $n_0 \in \mathbb N$ such that 
        \[
            H^1(X,W\O_X(-D)) \cong H^1(X,W_n\O(-D)) 
        \]
        for all $n_0 \leq n$.
        In particular, $H^1(X,W\O_X(-D))_\mathbb Q = 0$.
    \end{lemma}
    \begin{proof}
        Let $n_0$ as in Lemma \ref{lemSurfaceNaBSerre}, and $n_0 \leq t$, and consider the following part of the long exact cohomology sequence:
        \[
            0=H^1(X, W\O_X(-p^{t}D)) \xrightarrow {V^{t}} H^1(X,W\O_X(-D)) \rightarrow H^1(X,W_{t}\O_X(-D))
        \]
        We observe that
        \[
            h^1(X,W\O_X(-D)) \leq h^1(X,W_t\O_X(-D)).
        \]
        Similarly,
        \[
            0 = H^1(X,W_{t-t_0}\O_X(-p^{t_0}D)) 
            \xrightarrow{V^{t_0}} H^1(X,W_t\O_X(-D))
            \rightarrow H^1(X,W_{t_0}\O_X(-D))
        \]
        is exact.
        So we have
        \[
            0 \leq h^1(X,W\O_X(-D)) \leq h^1(X,W_t\O_X(-D)) \leq h^1(X,W_{t_0}\O(-D))
        \]
        for all $t_0 \leq t$,
        and since the latter is finite there exists a $t_1$ for which $h^1(X,W_t\O_X(-D))$ is smallest.
        Then
        \[
            h^1(X,W\O_X(-D)) = h^1(X,W_{t_1}\O(-D)).
        \]
    \end{proof}

    \begin{proposition}\label{propSurfaceNaBVtor}
        For $0<<n$, 
        \[
            \text{ker}\left(H^2(X,W\O_X(-p^{n}D)) \xrightarrow {V^s} H^2(X,W\O_X(-p^{n-s}D))\right) =0
        \] 
        for all $0<s\leq n$.
    \end{proposition}
    \begin{proof}
        Let $n_0$ be the integer from the above lemma 
        and $n_0 < n$.
        We have the following exact sequence
        \begin{align*}
            0 &\rightarrow H^1(X,W\O_X(-D)) \xrightarrow{\sim} H^1(X,W_{n}\O_X(-D)) \\
            &\xrightarrow 0 H^2(X,W\O_X(-p^{n}D)) \xrightarrow{V^n} H^2(X,W\O_X(-D)).
        \end{align*}
    \end{proof}

    From Theorem \ref{thm}, Lemma \ref{lemSurfaceNaB}, Proposition \ref{propSurfaceNaBVtor}, it immediately 
    follows the following Kawamata-Viehweg type vanishing theorem for Witt divisorial sheaves of 
    nef and big divisors.

    \begin{theorem}\label{thmKVV}
        For $X$ and $D$ as above,
        \renewcommand{\theenumi}{\roman{enumi}}
        \begin{enumerate}
            \item $H^1(X,W\O_X(-D))_\mathbb Q=0$,
            \item $H^1(X,\sHom_{W\O_X}(W\O_X(-D),\WO)) = 0$.
        \end{enumerate}
    \end{theorem}

    \section{Counter example}\label{sec4}
    There are examples of surfaces for which Ramanujam vanishing, i.e. the vanishing of $H^1(X,L^{-1})$ for nef and big $L$, fails (cf. \cite{Mukai}).
    Yet for Witt divisorial sheaves, the vanishing from section \ref{sec3} holds.
    Similarly, there are counterexamples of a logarithmic Kawamata--Viehweg type vanishing for surfaces (cf. \cite{CT}).
    It is then a natural question whether such a logarithmic Kawamata--Viehweg type vanishing holds for Witt divisorial sheaves.
    Below we compute an example to see that the answer is no.

    In \cite{CT} Cascini and Tanaka show that a number of rational surfaces, originally constructed by Langer \cite{Langer}, violate logarithmic Kawamata-Viehweg vanishing. In particular this includes a weak Del Pezzo surface in characteristic 2. 

    Let $q = p^e$ and $k$ be a field containing $\mathbb F_q$.
    Cascini and Tanaka \cite{CT}*{Notation 2.1} define the surface $X$ to be the base change to $k$ of the blowup of $\mathbb P_{\mathbb F_p}^2$ at its $q^2+q+1$ $\mathbb F_q$--rational points.
    $B$ and $\Delta$ are then defined as follows: 
    \begin{itemize}
        \item $B := (q^2 + 1)f^*H - q\Sigma_{i=1}^{q^2+q+1} E_i,$
        \item $\Delta := \frac{q}{q+1}\Sigma_{i=1}^{q^2+q+1} L_i',$
    \end{itemize}
    where $H$ is the hyperplane divisor on $\mathbb P_k^2$, and the $L_i'$ are the proper transforms of the $\mathbb F_q$--lines on $\mathbb P^2_{\mathbb F_q}$, pulled back to $X$ along the base change map.
    Then (cf. \cite{CT}*{Theorem 3.1})
    \begin{itemize}
        \item $E_i^2 = -1,$
        \item $(L_i')^2 = -q$,
        \item $(X,\Delta)$ is klt,
        \item $B-\Delta$ and $-K_X$ are nef and big,
        \item $h^1(X,\O_X(-B)) \geq \frac{1}{2}(q^2+q)$.
    \end{itemize}
    We shall attempt to calculate $H^1(X,W\O_X(-B))$ in the case of $q=p=2$ by calculating the cohomology of the truncated sheaves $W_n\O_X(-B)$ for all $n$ and their transition maps to obtain the limit.

    \begin{theorem}\label{example}
        Let $q=p=2$. Then 
        \[
            H^1(X,W\O_X(-B))_\mathbb Q \neq 0.
        \]
    \end{theorem}
    \begin{proof}
    Push forward along $f$ the following exact sequence:
    \begin{equation}\label{exactSequenceVerschiebung}
        0 \rightarrow F_*\O_X(-p^{n-1}B) \xrightarrow {V^{n-1}} W_n\O_X(-B) \xrightarrow R \O_X(-B) \rightarrow 0.
    \end{equation}
        Since $R^2f_*\O(-p^tB)=0$, for each term we obtain a Leray spectral sequence which degenerates at page three. 
        Thanks to functoriality then we have a three term complex of five term exact sequences. 
        For each term, the five term exact sequence is
    \begin{equation}\label{exactSequenceLeray}
    \begin{aligned}
        0 &\rightarrow H^1(\mathbb P^2,f_*W_k\O_X(-p^lB)) \rightarrow H^1(X,W_k\O_X(-p^lB)) \\
        &\rightarrow H^0(\mathbb P^2,R^1f_*W_k\O_X(-p^lB)) \rightarrow H^2(\mathbb P^2,f_*W_k\O_X(-p^lB)) \\
        &\rightarrow H^2(X,W_k\O_X(-p^lB)).
    \end{aligned}
    \end{equation}
    \begin{remark}
        It follows from the seven term exact sequence for spectral sequences that the final arrow is in fact a surjection. 
        However, we don't use this fact in the rest of the proof.
    \end{remark}
    We calculate the first, third, fourth and fifth terms to obtain the second.
    To calculate the limit we then need to determine the transition maps
    \[
        H^1(X,W_k\O_X(-p^lB)) \xrightarrow {R_n} H^1(X,W_{k-1}\O_X(-p^lB))
    \]
    \proofstep{First term}
    We have $f_*W_n\O_X(-B) = W_n\O_{\mathbb P^2}(-5)$, and hence 
    \[
        h^1(\mathbb P^2,f_*W_n\O_X(-p^tB)) = 0 \fa 0 \leq t.
    \]

    \proofstep{Third term}
    Since $R^1f_*\O_X(-B)$ is supported on the seven $k$--points and the $E_i$ have multiplicity 2 in $-B$, it is a skyscraper sheaf with global sections $k^7$ and zero higher cohomology.
    Similarly, $H^0(\mathbb P^2, R^1f_*\O_X(-p^nB)) = k^{7t_{2p^n}}$ where $t_l = \frac{l(l-1)}{2}$.
    We chase the derived push forward along $f$ of the exact sequence
    \[
        0 \rightarrow \O_X((n-1)E_i) \rightarrow \O_X(nE_i) \rightarrow \O_{E_i}(-n) \rightarrow 0
    \]
    to see this.
    Pushing forward along $f$ we obtain
    \[
        0 \rightarrow R^1f_*\O_X((n-1)E_i) \rightarrow R^1f_*\O_X(nE_i) \rightarrow R^1f_*\O_{E_i}(-n) \rightarrow 0.
    \]
    Away from its corresponding $\mathbb F_q$--rational point $P_i$, $R^1f_*\O_{E_i}(-n)$ is zero.
    On any open subset containing $P_i$, its sections are isomorphic to $H^0(\mathbb P^1,\O(-n))$, which is of dimension $n-1$ for $0<n$.
    Therefore $R^1f_*\O_X(nE_i)$ is a skyscraper sheaf supported on $P_i$, with 
    \[
        h^0(\mathbb P^2,R^1f_*\O_X(E_i^n)) 
        = h^0(\mathbb P^2,R^1f_*\O_X(E_i^{n-1})) + n - 1.
    \]
    I.e. its dimension is $\frac{n(n-1)}{2}$.
    There are seven $\mathbb F_q$--rational points $P_i$, and so seven corresponding $E_i$ in $-B$.
    The rest of $-B$ does not contribute because it has zero intersection with $E_i$.
    Hence the statement
    \[
        h^0(\mathbb P^2,R^1f_*\O_X(-p^nB)) = 7t_{2p^n}
    \]
    follows.

    To understand the Witt module structure of the Witt divisorial sheaf cohomology we need to understand how $V$ and $F$ act on $R^1f_*\O_X(-p^nB)$. 
    First, since 
    \[
        f_*W_n\O_X(-p^tB) \xrightarrow{R_n} f_*W_{n-1}\O_X(-p^tB)
    \] is naturally surjective, 
    \[
        R^1f_*\O_X(-p^{t+n-1}B) \xrightarrow{V^{n-1}} R^1f_*W_n\O_X(-p^tB)
    \]
    is injective.

    To understand $F$, take the short exact sequence (cf. \cite{Mukai})
    \begin{equation}\label{exactSequenceFrobenius}
        0 \rightarrow \O_{X} \xrightarrow{F} F_*\O_X \rightarrow \mathscr B_X \rightarrow 0
    \end{equation}
    tensor with $\O_X(-p^nB)$, and push forward along $f$.
    Outside of the seven $k$--points of $\mathbb P^2_k$, $f$ is an ismorphism, and so the restriction of $f_*\mathscr B_X$ to that open subset $U$ is just $(\mathscr B_{\mathbb P^2} \tensor \O_{\mathbb P^2}(-p^n5))|_U$.
    Therefore we just need to determine the stalks on the seven $k$--points.
    Let $E := E_i$ for any $i$, and $E_i \xrightarrow{i} X$ the inclusion.
    Applying $i^*$ to \ref{exactSequenceFrobenius} we obtain
    \[
        \O_E \xrightarrow{F} F_*\O_E \rightarrow (\mathscr B_X)|_E \rightarrow 0.
    \]
    We have $F_*\O_E \cong \O_E \oplus \O_E(-1)$ (cf. \cite{Thomsen}, for example).
    Since $F$ is non-zero on global sections, it follows that $(\mathscr B_X)|_E \cong \O_E(-1)$.

    Because $-B.E=-2$, repeating the entire process after tensoring with $\O_X(-p^nB)$ yields 
    \[
        (\mathscr B_X \tensor \O_X(-p^nB))|_E \cong \O_E(-1-2p^n),
    \]
    which has no global sections.
    Therefore, 
    \[
        R^1f_*\O_X(-p^nB) \xrightarrow{F} R^1f_*\O_X(-p^{n+1}B)
    \]
    is injective. 

    Finally, to compute the limit, we need to determine the transition maps $R_n$.
    Since $R^2f_*\O_X(-p^tB)$ is zero for all $t$, 
    \[
        R^1f_*W_n\O_X(-p^tB) \xrightarrow{R_n} R^1f_*W_{n-1}\O_X(-p^tB)
    \]
    is surjective for all $t$.
    Therefore, $H^0(\mathbb P^2, R^1f_*W_n\O_X(-B))$ is a direct sum of $W^T_t$ for $1\leq t \leq n$, certain $T$ depending on $t$ and $n$, including $(T,t) = (7,n)$.

    \proofstep{Fourth term}
    Using Serre duality we can compute 
    \[
        \begin{aligned}
            h^2(\mathbb P^2,f_*\O_X(-B)) &= h^0(\mathbb P^2,\O(2)) = 6,\\
            h^2(\mathbb P^2,f_*\O_X(-p^tB)) &= h^0(\mathbb P^2,\O(p^t5-3)) = \binom{2+p^t5-3}{p^t5-3} \fa 1 \leq t.
        \end{aligned}
    \]
    Now, $H^2(\mathbb P^2,W_n\O_{\mathbb P^2}(-5))$ sits in the following exact sequence of $W_n$--modules:
    \[
        \begin{aligned}
            \cdots &\rightarrow H^2(\mathbb P^2,\O_{\mathbb P^2}(-p^{n-1}5)) \xrightarrow {V^{n-1}} H^2(\mathbb P^2,W_n\O_{\mathbb P^2}(-5)) \\
            &\xrightarrow {R} H^2(\mathbb P^2,W_{n-1}\O_{\mathbb P^2}(-5)) \rightarrow 0.
        \end{aligned}
    \]
    To compute its structure as a $W_n$--module for $1 < n$ then, we need to know how $V$ and $F$ act on $H^2(\mathbb P^2,\O_{\mathbb P^2}(-p^n5))$ for any $n$.

    Since the $H^1$ are zero, $V$ is injective on $H^2$.
    Since $\mathbb P^2$ is $F$-split (cf. \cite{SZ15}*{Example 2.2}), $F$ is injective on $H^2$.
    It follows that similar to the third term, the fourth term is a direct sum of $W^T_t$, for certain $t$, $T$ depending on $t$ and $n$, including $(T,t)=(6,n)$.

    \proofstep{Fifth term}
    Using Serre duality we can easily compute $H^2(X,\O_X(-B))=0$.
    For $1<n$, the fifth term then depends on the image of the $H^1$ term under $R$, as it is the quotient $Q^{n-1}$ of $H^2(X,\O_X(-p^{n-1}B))$ by the cokernel of $R$, and so we cannot compute it before knowing the behavior of $R$ on the first cohomology.

    \proofstep{Taking the limit}
    Having computed the relevant terms as far as possible, sequence \ref{exactSequenceLeray} becomes a four term exact sequence, and the three term complex of (now four term) exact sequences induced by \ref{exactSequenceVerschiebung} yields the commutative diagram of $W_n$--modules
    \begin{equation*}
        \begin{tikzcd}
            0 \arrow{r} & H^1(X,\O_X(-p^{n-1}B)) \arrow{r} \arrow{d}{V^{n-1}} & A_1^{n-1} \arrow{r} \arrow{d}{V^{n-1}} & B_1^{n-1} \arrow{r} \arrow{d}{V^{n-1}} & Q_1^{n-1} \arrow{d}{V^{n-1}} \\
            0 \arrow{r} & H^1(X,W_n\O_X(-B)) \arrow{r} \arrow{d}{R_n} & A_n^0 \arrow{r}{\phi} \arrow{d}{R_n} & B_n^0 \arrow{r} \arrow{d}{R_n} & Q_n^0 \arrow{d}{R_n}\\
            0 \arrow{r} & H^1(X,W_{n-1}\O_X(-B)) \arrow{r}{\psi} & A_{n-1}^0 \arrow{r} & B_{n-1}^0 \arrow{r} & Q_{n-1}^0.
        \end{tikzcd}
    \end{equation*}
    Here we wrote 
    \[
        \begin{aligned}
            A^l_k &:= R^1f_*W_k\O_X(-p^lB),\\
            B^l_k &:= H^2(\mathbb P^2, W_k\O_X(-p^lB)),\\
            Q^l_k &:= H^2(X, W_k\O_X(-p^lB)).
        \end{aligned}
    \]
    for brevity.
    \begin{remark}
        In the case $n=2$ the bottom row gives another proof that
        \[
            H^1(X,\O_X(-B)) \neq 0.
        \]
        In fact its dimension is equal to one, because $Q_1^0$ is zero.
    \end{remark}
    As shown above, the $A_k^l$ and $B_k^l$ are direct sums of $W^T_k$ with $T$ an integer depending on $k$ and $l$.
    Both $R_n$ and $\phi$ are evaluated on each summand individually.
    Now let $e$ be a non-zero element in $H^1(X,W_{n-1}\O_X(-B))$.
    We already know that the second and third columns are in fact short exact sequences, with $V^{n-1}$ injective and $R_n$ surjective.
    Hence $\psi$ maps $e$ into $A_{n-1}^0$, where it has a non-zero preimage under $R_n$ in $A_n^0$.
    In fact, since $A_1^{n-1}$ is the kernel of $R_n$, $({R_n})^{-1}(\psi(e))$ has dimension 
    \[
        \text{dim}(A_1^{n-1}) = 7\left(\frac{(2p^{n-1}-1)2p^{n-1}}{2}\right).
    \]
    On the other hand, the image of $A_1^{n-1}$ in $B_n^0$ lies in $B_1^{n-1}$, which has dimension
    \[
        \text{dim}(B_1^{n-1}) = \binom{2+p^{n-1}5-3}{p^{n-1}5-3} = \frac{(2+p^{n-1}5-3)!}{2!(p^{n-1}5-3)!}.
    \]
    Comparing the two dimensions, we can see that for $0<n$,
    \[
        \begin{aligned}
            &\text{dim}\left(B_1^{n-1}\right) - \text{dim}\left(A_1^{n-1}\right)\\ 
            &= \frac{\left(2+p^{n-1}5-3\right)!}{2!\left(p^{n-1}5-3\right)!} 
                - 7\left(\frac{\left(2p^{n-1}-1\right)2p^{n-1}}{2}\right) \\
            &= \frac{\left(p^{n-1}5-1\right)!}{2!\left(p^{n-1}5-3\right)!}
                - 7 \left(p^{n-1}\left(2p^{n-1}-1\right)\right) \\
            &= \frac{\left(p^{n-1}5-1\right)\left(p^{n-1}5-2\right)}{2}
                - 7 \left(p^{n-1}\left(2p^{n-1}-1\right)\right) \\
            &= \frac{\left(25\cdot p^{2(n-1)} - 15\cdot p^{n-1} + 2\right)}{2}
                - 7\left(2\cdot p^{2(n-1)} - p^{n-1}\right) \\
            &= \frac{1}{2} \left((25-28)p^{2(n-1)} + \left(14-15\right)p^{n-1} + 2\right) \\
            &= \frac{1}{2} \left(-3\cdot p^{2(n-1)} - p^{n-1} + 2\right) < 0.
        \end{aligned}
    \]
    Evidently, for $0<n$,
    \[
        \text{dim}(B_1^{n-1}) < \text{dim}(A_1^{n-1}).
    \]
    That is, the preimage of $\psi(e)\in A_{n-1}^0$, which lies in $A_n^0$, has non-zero kernel under $\phi$.
    I.e. it, and so $e$, has a non-zero preimage in $H^1(X,W_n\O_X(-B))$. 

    Hence the transition maps $R_n$ on $H^1(X,W_n\O_X(-B))$ are surjective, and 
    \[
        H^1(X,W\O_X(-B))_\mathbb Q \neq 0.
    \]
    In fact it has infinite dimension over $W_\mathbb Q$.
    \end{proof}

\end{document}